\documentclass[11pt]{amsart}
 \usepackage{amssymb,amsthm}

\usepackage[]{graphicx}
\usepackage{psfrag}

\everymath{\displaystyle}

\newtheorem{theo}{Theorem}[section]
\newtheorem{prop}[theo]{Proposition}
\newtheorem{coro}[theo]{Corollary}
\newtheorem{lemma}[theo]{Lemma}
\newtheorem{teo}[theo]{Theorem}
\newtheorem{lem}[theo]{Lemma}
\newtheorem{cor}[theo]{Corollary}
\theoremstyle{definition}
\newtheorem{defi}[theo]{Definition}
\newtheorem{rema}[theo]{Remark}

\newtheorem{remark}[theo]{Remark}
\newtheorem{remarks}[theo]{Remarks}

\newcommand{\fin}{\hfill$\square$}

\newcommand{\ADS}{\mathbb{ADS}}

\newcommand{\wt}{\widetilde}

\newcommand{\Cl}{\mbox{Cl}}

\newcommand{\sD}{\mathfrak D}

\newcommand{\cQ}{\mathcal Q}

\def\d{\mathrm{d}}

\newcommand{\DD}{\mathbb{D}}

\newcommand{\HH}{\mathbb{H}}

\newcommand{\PP}{\mathbb{P}}

\newcommand{\RR}{\mathbb{R}}
\renewcommand{\SS}{\mathbb{S}}

\newcommand{\ZZ}{\mathbb{Z}}
\newcommand{\op}{\operatorname}

\DeclareMathOperator{\SO}{SO}
\DeclareMathOperator{\T}{T}
\DeclareMathOperator{\AdS}{AdS}
\DeclareMathOperator{\Ein}{Ein}
\DeclareMathOperator{\Rep}{Rep}

\newcommand{\uEin}{\widetilde{\Ein}}
\newcommand{\ie}{\emph{ie. }}

\newcommand{\sca}[2]{\langle #1 | #2\rangle}         
\newcommand{\tq}{\,;\,} 
\newcommand{\mcal}[1]{\ensuremath{\mathcal{#1}}}
\newcommand{\orth}{\bot}
\newcommand{\cY}{\mcal{Y}}

\newcommand{\cD}{\mcal{D}}
\newcommand{\cE}{\mcal{E}}
\newcommand{\cF}{\mcal{F}}

\newcommand{\cL}{\mcal{L}}
\newcommand{\cC}{\mcal{C}}
\newcommand{\cU}{\mcal{U}}

\begin{document}

\title[Anosov AdS representations are Quasi-Fuchsian]
{Anosov AdS representations are Quasi-Fuchsian}

\author[Q. M\'erigot]{Quentin M\'erigot}
\email{Quentin.Merigot@sophia.inria.fr} 
\address{UMPA, \'Ecole Normale Sup\'erieure de Lyon}
\curraddr{INRIA Sophia-Antipolis}


\begin{abstract}
Let $\Gamma$ be a cocompact lattice in $\SO(1,n)$.
A representation $\rho: \Gamma \to \SO(2,n)$ is
\textit{quasi-Fuchsian} if it is faithfull, discrete, and preserves an
acausal subset in the boundary of anti-de Sitter space - a particular
case is the case of \textit{Fuchsian representations}, \ie composition
of the inclusions $\Gamma \subset \SO(1,n)$ and $\SO(1,n) \subset
\SO(2,n)$. We prove that if a representation  is Anosov in the sense of
Labourie (cf. \cite{labourie}) then it is also
quasi-Fuchsian. We also show that Fuchsian representations are Anosov
: the fact that all quasi-Fuchsian representations are Anosov
will be proved in a second part by T. Barbot. The study 
involves the geometry of locally  anti-de Sitter spaces:
quasi-Fuchsian representations are holonomy representations of
globally hyperbolic spacetimes diffeomorphic to $\RR \times
\Gamma\backslash\HH^{n}$ locally modeled on $\AdS_{n+1}$. 
\end{abstract}

\maketitle


\section{Introduction}
Let $\SO_0(1,n)$, $\SO_0(2,n)$ denote the identity components of respectively
$\SO(1,n)$, $\SO(2,n)$.
Let $\Gamma$ be a cocompact torsion free lattice in $\SO_{0}(1,n)$. 
For any Lie group $G$ let $\op{Rep}(\Gamma, G)$ denote the space of representations
of $\Gamma$ into $G$ equipped with the compact-open topology. 

In the case $G=\SO_{0}(1,n+1)$ we distinguish the \textit{Fuchsian representations:} they are the representations
obtained by composition of an embedding $\Gamma \subset \SO_{0}(1,n)$ (by Mostow rigidity,
there is only one up to conjugacy if $n \geq 3$) and any faithfull representation of $\SO_{0}(1,n)$ into $\SO_{0}(1,n+1)$.
Their characteristic property is to be faithfull, discrete, and to preserve a totally geodesic copy of $\HH^{n}$ into $\HH^{n+1}$.

If we relax the last condition by only requiring the existence of a $\rho(\Gamma)$-invariant topological
$(n-1)$-sphere in $\partial\HH^{n+1}$ (in the Fuchsian case, the boundary of the $\rho_{0}(\Gamma)$-invariant totally geodesic hypersurface
$\HH^{n} \subset \HH^{n+1}$ provides such a topological sphere), we obtain the notion of
\textit{quasi-Fuchsian representation.} We denote 
by $\cQ\cF(\Gamma, \SO_{0}(1,n+1))$ the set of quasi-Fuchsian representations.
It is well-known that $\cQ\cF(\Gamma, \SO_{0}(1,n+1))$ is a neighborhood of 
Fuchsian representations in the space of representations of $\Gamma$ into $\SO_{0}(1,n+1)$. One way to prove
this assertion, based on the Anosov character of the geodesic flow $\phi^{t}$ of the hyperbolic manifold $N=\Gamma\backslash\T^{1}\HH^{n}$
(for definitions, see \S~\ref{sub.defsgenerales}) goes as follows: $\phi^{t}$ is the projection of the geodesic flow
$\tilde{\phi}^{t}$ on $\T^{1}\HH^{n}$. For every $(x,v)$ in $\T^{1}\HH^{n}$, let $\ell^{+}(x,v)$, $\ell^{-}(x,v)$ be
the extremities in $\partial\HH^{n} \subset \partial\HH^{n+1}$ of the unique geodesic tangent to $(x,v)$. These
maps define an $\rho_{0}$-equivariant map 
$(\ell^{+}, \ell^{-}): \T^{1}\HH^{n} \to \partial\HH^{n+1} \times \partial\HH^{n+1} \setminus \sD$
where $\sD$ is the diagonal. To any $x$ in $\HH^{n}$ attach a metric $g^{x}$ on $\partial\HH^{n+1}$ varying with $x$ continuously 
and in a $\Gamma$-equivariant way - for example, take the angular metric at $x$, \ie the pull-back of the natural metric on
$\T^{1}_{x}\HH^{n+1}$ by the map associating to a point $p$ in $\partial\HH^{n+1}$ the unit tangent vector at $x$
of the geodesic ray starting from $x$ and ending at $p$. This family of metrics satisfies the following property: \textit{given $p$
in $\partial\HH^{n}$ and a tangent vector $w$ to $\partial\HH^{n+1}$ at $p$, the norm $g^{x^{t}}(w)$ increases exponentially 
with $t$ when $x^{t}$ describes a geodesic ray with final extremity $p$.} This property has the following consequence:
consider the flat bundle 
$E_{\rho_{0}}=\Gamma\backslash(\T^{1}\HH^{n} \times (\partial\HH^{n+1} \times \partial\HH^{n+1} \setminus \sD))$
associated to the representation $\rho_{0}$. Denote by $\pi_{\rho_{0}}: E_{\rho_{0}} \to \Gamma\backslash\T^{1}\HH^{n}$
the natural fibration. The map $(\ell^{+}, \ell^{-})$ defined above induces 
a section $s_{\rho_{0}}$ of $\pi_{\rho_{0}}$. The flow $\Phi^{t}(x,v, \ell^{+}, \ell^{-})=(\tilde{\phi}^{t}(x,v), \ell^{+}, \ell^{-})$
induces a flow $\phi_{\rho_{0}}^{t}$ on $E_{\rho_{0}}$ that preserves the image of $s_{\rho_{0}}$. Last but not least,
the existence of the metrics $g^{x}$ ensures that \textit{as a $\phi_{\rho_{0}}^{t}$-invariant closed subset of $E_{\rho_{0}}$,
the image of $s_{\rho_{0}}$ is a $\phi_{\rho_{0}}^{t}$-hyperbolic set} (cf. \S~\ref{sub.defsgenerales}).
When we deform $\rho_{0}$, the flat bundle and the flow $\phi^{t}_{\rho_{0}}$ can be continuously deformed.
The structural stability of hyperbolic invariant closed subsets ensures that for small deformations we still have 
a section $s_{\rho}$ of the flat bundle $E_{\rho}$, the image of which is $\phi^{t}_{\rho}$-hyperbolic. 
This section lifts to an equivariant map 
$\ell_{\rho}=(\ell_{\rho}^{+}, \ell^{-}_{\rho}): \T^{1}\HH^{n} \to \partial\HH^{n+1} \times \partial\HH^{n+1} \setminus \sD$.
It is quite straightforward to observe that $\ell^{+}_{\rho}$ must be constant along the stable leaves of the geodesic flow,
\ie the fibers of $\ell^{+}$. Therefore, it induces a continuous map $\bar{\ell^{+}}: \partial\HH^{n} \to \partial\HH^{n+1}$,
the image of which is the a $\rho(\Gamma)$-invariant topological $(n-1)$-sphere in $\partial\HH^{n+1}$.

This kind of argument has been extended in a more general framework by F. Labourie in \cite{labourie}:
he defined, for any pair $(G,Y)$ where $G$ is a Lie group acting on a manifold $Y$, the notion of 
\textit{$(G,Y)$-Anosov representation} (or simply Anosov representation when there is no ambiguity about
the pair $(G,Y)$). For a definition, see~\ref{sub.defsgenerales}. We denote by $\op{Anos}_{Y}(\Gamma, G)$
the space of $(G,Y)$-Anosov representations. By structural stability, $\op{Anos}_{Y}(\Gamma, G)$ is an open domain,
and simple, general arguments ensure that Anosov representations are faithfull, with discrete image formed
by loxodromic elements. As a matter of fact, $\cQ\cF(\Gamma, \SO_{0}(1,n+1))$ and $\op{Anos}_{Y}(\Gamma, \SO_{0}(1,n+1))$
where $Y=\partial\HH^{n+1} \times \partial\HH^{n+1} \setminus \sD$ coincide: we sketched above a proof
of one implication, but observe that the reverse implication, namely that quasi-Fuchsian representations are
Anosov, is less obvious (it can be obtained by adapting the arguments
given in the case $G=\SO_{0}(2,n)$ in T. Barbot's sequel to this
article \cite{anosovads2}). 

Anosov representations have been studied in different situations, mostly in the case $n=2$, \ie the case
where $\Gamma$ is a surface group:

-- in \cite{labourie}, F. Labourie proved that when $G$ is the group $\op{SL}(n, \RR)$ and $Y$ the frame variety,
one connected component of $\op{Anos}_{Y}(\Gamma, G)$, the \textit{quasi-Fuchsian component,} coincides with
a connected component of $\op{Rep}(\Gamma, G)$: the Hitchin component. Moreover, he proved that these quasi-Fuchsian representations
are \textit{hyperconvex}, \ie that they preserve some curve in the projective space $\PP(\RR^{n})$ with some very strong
convexity properties. In \cite{guichard}, O. Guichard then proved that conversely hyperconvex representations
are quasi-Fuchsian. Beware: $(G,Y)$-Anosov representations are not necessarily quasi-Fuchsian; in other words, 
$\op{Anos}_{Y}(\Gamma, G)$ is not connected. See~\cite{barflag}.

-- In \cite{burlab}, the authors also used the notion of Anosov representations for the study of representations 
of surface groups into the symplectic group of a real symplectic vector space with maximal Toledo invariant.

The present paper is devoted to the case where $\Gamma$ is a cocompact lattice of $\SO_{0}(1,n)$ that we deform in $G =\SO_{0}(2,n)$.
Whereas in the case of quasi-Fuchsian representations into $\SO_{0}(1,n+1)$ presented above the geometry of hyperbolic space
$\HH^{n+1}$ played an important role, the study of $\op{Rep}(\Gamma, \SO_{0}(2,n))$ deeply involves the geometry of
the Lorentzian analog of $\HH^{n+1}$, namely the anti-de Sitter space $\AdS_{n+1}$. In Lorentzian geometry,
appear some phenomena, latent in the Riemannian context, related to the causality notions. Whereas in hyperbolic
space pair of points are only distinguished by their mutual distance, in the anti-de Sitter space we have to distinguish
three types of pair of points, according to the nature of the geometry joining the two points: this geodesic may be 
spacelike, lightlike or timelike - in the last two cases, the points are said \textit{causally related.}

The conformal boundary $\partial\HH^{n+1}$ of the hyperbolic space plays an important role. Similarly,
anti-de Sitter space admits a conformal boundary: the \textit{Einstein universe} $\Ein_n$. It is a
conformal Lorentzian spacetime, also subject to a causality notion. 
In the following theorem, $\cY$ is the subset of $\Ein_{n} \times \Ein_{n}$
made of non-causally related pairs, \ie pairs of points that can be joined by a spacelike geodesic in $\AdS_{n+1}$:

\begin{theo}
\label{teo:maintheorem}
Any $(\SO_{0}(2,n), \cY)$-Anosov representation  $\rho: \Gamma \to
\SO_{0}(2,n)$ is quasi-Fuchsian.
\end{theo}

The geometric ingredient of this Theorem is the fact that
quasi-Fuchsian representations are precisely holonomy representations
of Lorentzian manifolds locally modelled on $\AdS_{n+1}$ which are
\textit{spatially compact, globally hyperbolic} (in short, GHC)
(\S~\ref{sec.basic}). In this introduction, let's simply mention that,
among many others, a characterization of these spacetimes is the fact
to admit  a proper time function (a time function being a function
with everywhere timelike gradient). It is only recently that the
relevance of this notion in constant curvature spacetimes started to
be perceived, a great impetus being given by the paper \cite{mess1}
when it was circulating in the physical and the mathematical community
as well in the 90's (see also \cite{mess2}). The classification of GHC
spacetimes of constant curvature $-1$ is one of the main motivation of
the present paper (the case of constant curvature $+1$ and $0$ being
already treated in respectively \cite{scannell}, \cite{barflat}), and
of its sequel by T. Barbot (\cite{anosovads2}) where the converse of
theorem \ref{teo:maintheorem} is proved.

\newpage

\section{Preliminaries}
\subsection{Basic causality notions}
\label{sec.basic}
We assume the reader acquainted to basic causality notions in Lorentzian manifolds
like \emph{causal} or \emph{timelike} curves, \emph{inextendible} causal curves,
\emph{time orientation,\/} \emph{future} and
\emph{past} of subsets, \emph{time function,} \emph{achronal} subsets, etc... 
We refer to \cite{beem} or \cite[\S~14]{oneill} for further details.

By \emph{spacetime} we mean here an oriented and time oriented manifold. 
A spacetime is \emph{strongly causal} if its topology admits a basis of \emph{causally convex} neighborhoods,
\ie neighborhoods $U$ such that any causal curve with extremities in $U$ is contained in $U$.

Recall that a spacetime $(M,g)$ is \emph{globally hyperbolic} (abbreviation GH)
if it admits a \emph{Cauchy hypersurface}, \ie
an achronal subspace $S$ which intersects every inextendible timelike
curve at exactly one point - such a subspace is automatically a topological locally Lipschitz
hypersurface (see \cite[\S~14, Lemma 29]{oneill}).
 
A globally hyperbolic spacetime is called
\emph{spatially compact} (in short GHC) if its Cauchy hypersurfaces are compact. 
An alternative and equivalent definition of GHC spacetimes is to require
the existence of a \emph{proper} time function.

\subsection{Anti-de Sitter space}
Let $\RR^{2,n}$ be the vector space of dimension $n+2$, with coordinates 
$(u,v, x_{1}, \ldots , x_{n})$, endowed with the quadratic
form:
\[ \mathrm{q}_{2,n}:=- u^{2} - v^{2} + x_{1}^{2} + \ldots + x_{n}^{2} \]

We denote by $\sca{x}{y}$ the associated scalar product. For any subset $A$ of
$\RR^{2,n}$ we denote $A^{\orth}$ the orthogonal of $A$, \ie the set
of elements $y$ in $\RR^{2,n}$ such that $\sca{y}{x}=0$ for every $x$ in $A$.
We also denote by $\cC_{n}$ the isotropic cone $\{ \mathrm{q}_{2,n}=0 \}$.

\begin{defi}
The anti-de Sitter space $\AdS_{n+1}$ is $\{ \mathrm{q}_{2,n}=-1 \}$
endowed with the Lorentzian metric obtained by restriction of $\mathrm{q}_{2,n}$.
\end{defi}

We will also consider the coordinates $(r, \theta, x_{1}, ... , x_{n})$ with:
\[ u=r\cos(\theta), v=r\sin(\theta)\]

Observe the analogy with the definition of hyperbolic space $\HH^{n}$ - moreover,
every subset $\{ \theta=\theta_{0} \}$ is a totally
geodesic copy of the hyperbolic space embedded in $\AdS_{n+1}$. More generally,
the totally geodesic subspaces of dimension $k$ in
$\AdS_{n+1}$ are connected components of the intersections of $\AdS_{n+1}$ with
the linear subspaces of dimension $(k+1)$ in $\RR^{2,n}$. In particular, 
geodesics are intersections with $2$-planes.


\begin{remark}
\label{rk.norme-euclidienne}
We will also often need an auxiliary Euclidean metric on
$\RR^{2,n}$. Let's fix once for all the euclidean norm $\Vert_{0}$
defined by: 

\[ \Vert (u, v, x_{1}, \ldots, x_{n}) \Vert_{0}^{2}:=u^{2} + v^{2} + x_{1}^{2} + ... + x_{n}^{2} \]

\end{remark}

\subsection{Conformal model }

\begin{prop}
  \label{p.causal-structure}
The anti-de Sitter space $\AdS_{n+1}$ is conformally equivalent to
$(\SS^1\times\DD^{n},-d\theta^2+ds^2)$, where $d\theta^2$ is the standard riemannian
metric on $\SS^1=\RR/2\pi\ZZ$, where $ds^2$ is the standard metric (of
curvature $+1$) on the sphere $\SS^{n}$ and $\DD^{n}$ is the open upper
hemisphere of $\SS^{n}$.
\end{prop}

\begin{proof}
In the $(r, \theta, x_{1}, ... , x_{n})$-coordinates the $\AdS$ metric is:

\[ -r^{2}\op{d\theta}^{2} + \op{ds}_{hyp}^{2} \]

where $\op{ds}_{hyp}^{2}$ is the hyperbolic norm, \ie the induced metric on $\{ \theta=\theta_{0} \} \approx \HH^{n}$.
More precisely, $\{ \theta=\theta_{0} \}$ is a sheet of the hyperboloid $\{ -r^{2} + x^{2}_{1} + ... + x^{2}_{n}=-1 \}$.
The map $(r, x_{1}, \ldots , x_{n}) \to (1/r, x_{1}/r, \ldots, x_{n}/r)$ sends this hyperboloid on $\DD^{n}$, and an easy computation shows
that the pull-back by this map of the standard metric on the hemisphere is $r^{-2}\op{ds}_{hyp}^{2}$.
The proposition follows.
\end{proof}

Proposition~\ref{p.causal-structure} shows in particular that $\AdS_{n+1}$
contains many closed causal curves. But the universal covering 
$\wt\AdS_{n+1}$, conformally
equivalent to $(\RR\times\DD^{n},-d\theta^2+ds^2)$, contains no periodic causal curve.
It is strongly causal, but not globally hyperbolic.

\subsection{Einstein universe} Einstein universe $\Ein_{n+1}$ is the product $\SS^{1} \times \SS^{n}$
endowed with the metric $-d\theta^{2} + ds^{2}$ where $ds^{2}$ is as above the standard spherical metric. The universal
Einstein universe $\wt\Ein_{n+1}$ is the cyclic covering $\RR \times \SS^{n}$ equipped with the
lifted metric still denoted $-d\theta^{2} + ds^{2}$, but where $\theta$ now takes value in $\RR$. 
According to this definition, $\Ein_{n+1}$ and $\wt\Ein_{n+1}$ are Lorentzian manifolds, but it is
more adequate to consider them as conformal Lorentzian manifolds. 
We fix a time orientation: the one for which the coordinate $\theta$ is a time function on $\wt\Ein_{n+1}$.

In the sequel, we denote by $\mathrm{p}: \wt\Ein_{n+1} \to \Ein_{n+1}$ the
cyclic covering map. Let $\delta: \wt\Ein_{n+1} \to \wt\Ein_{n+1}$ be a generator of the Galois group of this 
cyclic covering. More precisely, we select $\delta$ so that for any $\tilde{x}$ in
$\wt\Ein_{n+1}$ the image $\delta(\tilde{x})$
is in the future of $\tilde{x}$.

Even if Einstein universe is merely a conformal Lorentzian spacetime, one can define
the notion of \textit{photons,} \ie (non parameterized) lightlike geodesics. We can also 
consider the causality relation in $\Ein_{n+1}$ and $\wt\Ein_{n+1}$. In particular, we define 
for every $x$ in $\Ein_{n+1}$ the \textit{lightcone} $C(x)$: it is the union of photons containing
$x$. If we write $x$ as a pair $(\theta, \mathrm{x})$ in $\SS^1 \times \SS^n$, the lightcone $C(x)$ is the set
of pairs $(\theta', \mathrm{y})$ such that $|\theta'-\theta|=d(\mathrm{x},\mathrm{y})$ where $d$ is distance
function for the spherical metric $ds^{2}$.

There is only one point in $\SS^n$ at distance $\pi$ of $\mathrm{x}$: the antipodal point $\mathrm{x}^\ast$.
Above this point, there is only one point in $\Ein_{n+1}$ contained in $C(x)$:
the antipodal point $x^\ast=(\theta+\pi, \mathrm{x}^\ast)$. The lightcone $C(x)$ with the points
$x$, $x^\ast$ removed is the union of two components:

-- the \textit{future cone:} it is the set $C^+(x):=\{ (\theta',\mathrm{y}) / \theta < \theta' < \theta+\pi, \; d(\mathrm{x},\mathrm{y})=\theta'-\theta \}$,

-- the \textit{past cone:} it is the set $C^-(x):=\{ (\theta',\mathrm{y}) / \theta-\pi < \theta' < \theta, \; d(\mathrm{x},\mathrm{y})=\theta-\theta' \}$.

Observe that the future cone of $x$ is the past cone of $x^\ast$, and that the past cone
of $x$ is the future cone of $x^\ast$.

According to Proposition~\ref{p.causal-structure} $\AdS_{n+1}$ (respectively $\wt\AdS_{n+1}$) conformally embeds in
$\Ein_{n+1}$ (respectively $\wt\Ein_{n+1}$). Hence the time orientation on $\Ein_{n+1}$
selected above induces a time orientation on $\AdS_{n+1}$ and $\wt\AdS_{n+1}$.
Since the boundary $\partial\DD^{n}$
is an equatorial sphere, the boundary $\partial\wt\AdS_{n+1}$ is a copy of the Einstein universe 
$\wt\Ein_{n}$. In other words, one can attach a ``Penrose boundary'' $\partial\wt\AdS_{n+1}$ to
$\wt\AdS_{n+1}$ such that $\wt\AdS_{n+1}\cup\partial\wt\AdS_{n+1}$ is conformally
equivalent to $(\SS^1\times\overline{\DD}^{n},-d\theta^2+ds^2)$, where
$\overline{\DD}^{n}$ is the closed upper hemisphere of $\SS^{n}$.

The restrictions of $\mathrm{p}$ and $\delta$ to $\wt\AdS_{n+1} \subset \wt\Ein_{n+1}$
are respectively a covering map over $\AdS_{n+1}$ and a generator of the Galois group
of the covering; we will still denote them by $\mathrm{p}$ and $\delta$.

\subsection{Isometry groups}
Every element of $\SO(2,n)$ induces an isometry of $\AdS_{n+1}$, and, for $n \geq 2$,
every isometry of $\AdS_{n+1}$ comes from an element of $\SO(2,n)$. Similarly, 
conformal isometries of $\Ein_{n+1}$ are projections of elements of $\SO(2,n+1)$ acting
on $\cC_{n+1}$ (still for $n \geq 2$).

In the sequel, we will only consider isometries preserving the orientation and the time orientation,
\ie elements of the neutral component $\SO_{0}(2,n)$ (or $\SO_{0}(2,n+1)$).

\subsection{Achronal subsets}
\label{sub.achro}
Recall that a subset of a conformal Lorentzian manifold is achronal (respectively acausal) if there is
no timelike (respectively causal) curve joining two distinct points of the subset.
In $\Ein_{n}\approx
(\RR\times\SS^{n-1},-d\theta^2+ds^2)$, every achronal subset is precisely the graph
of a $1$-Lipschitz function $f: \Lambda_0 \rightarrow {\mathbb R}$ where
$\Lambda_0$ is a subset of ${\mathbb S}^{n-1}$ endowed with its canonical
metric $d$). In particular,
the achronal closed topological hypersurfaces in $\partial\wt\AdS_{n+1}$ are
exactly the graphs of the $1$-Lipschitz functions $f:\SS^{n-1}\to\RR$: they are
topological $(n-1)$-spheres.

Similarly, achronal subsets of $\wt\AdS_{n+1}$ are graphs of $1$-Lipschitz functions
 $f: \Lambda_0 \rightarrow {\mathbb R}$ where $\Lambda_0$ is a subset of ${\mathbb D}^{n}$,
 and achronal topological hypersurfaces are graphs of $1$-Lipschitz maps $f: \DD^{n} \rightarrow {\mathbb R}$.

\emph{Stricto-sensu,\/} there is no achronal subset in $\Ein_{n+1}$ since closed timelike curves 
through a given point cover the entire
$\Ein_{n+1}$. Nevertheless, we can keep track of this notion in $\Ein_{n+1}$ by defining ``achronal'' subsets
of $\Ein_{n+1}$ as projections of geniune achronal subsets of $\wt\Ein_{n+1}$. This definition is
justified by the following results:

\begin{lem}
\label{le.achrinj}
The restriction of $\mathrm{p}$ to any achronal subset of $\wt\Ein_{n+1}$ is injective.
\end{lem}

\begin{proof}
Since the diameter of $\SS^{n}$ is $\pi$, the difference between the $t$-coordinates of two elements
of an achronal subset of $\wt\Ein_{n+1}$ is at most $\pi$. The lemma follows immediately.
\end{proof}

\begin{cor}
\label{cor.lift}
Let $\wt\Lambda_{1}$, $\wt\Lambda_{2}$ be two achronal subsets of $\wt\Ein_{n+1}$
admitting the same projection in $\Ein_{n+1}$. Then there is an integer $k$ such that:
$$\wt\Lambda_{1}=\delta^{k}\wt\Lambda_{2}$$
\fin
\end{cor}

\subsection{The Klein model $\ADS_{n+1}$ of the anti-de Sitter
  space} 
We now consider the quotient $\SS(\RR^{2,n})$ of
$\RR^{2,n}\setminus\{0\}$ by positive homotheties. In other
words, $\SS(\RR^{2,n})$ is the double covering of the projective space
$\PP(\RR^{2,n})$. We denote by $\SS$ the projection of
$\RR^{2,n}\setminus\{0\}$ on $\SS(\RR^{2,n})$. 
The projection $\SS$ is one-to-one in restriction to
$\AdS_{n+1}=\{ \mathrm{q}_{2,n}=-1 \}$. The \emph{Klein model} $\ADS_{n+1}$ of the
anti-de Sitter space is the projection of $\AdS_{n+1}$ in $\SS(\RR^{2,n})$,
endowed with the induced Lorentzian metric. 

$\ADS_{n+1}$ is also the projection of the open domain of
$\RR^{2,n}$ defined by the inequality $\{ \mathrm{q}_{2,n}<0 \}$. 
The topological boundary of $\ADS_{n+1}$ in $\SS(\RR^{2,n})$ is the
projection of the isotropic cone $\cC_n=\{ \mathrm{q}_{2,n}=0 \}$; we will denote this
boundary by $\partial\ADS_{n+1}$. By construction, the projection $\SS$
defines an isometry between $\AdS_{n+1}$ and $\ADS_{n+1}$. 
The continuous extension of this isometry is a canonical
homeomorphism between 
$\AdS_{n+1}\cup\partial\AdS_{n+1}$ and $\ADS_{n+1}\cup\partial\ADS_{n+1}$. 

For every linear subspace $F$ of dimension $k+1$ in $\RR^{2,n}$, we
denote by $\SS(F):=\SS(F\setminus\{0\})$ the corresponding projective subspace of
dimension $k$ in $\SS(\RR^{2,n})$. The geodesics of $\ADS_{n+1}$ are the
connected components of the intersections of $\ADS_{n+1}$ with the
projective lines $\SS(F)$ of $\SS(\RR^{2,n})$. More generally, the
totally geodesic subspaces of dimension $k$ in $\ADS_{n+1}$ are the
connected components of the intersections of $\ADS_{n+1}$ with the
projective subspaces $\SS(F)$ of dimension $k$ of $\SS(\RR^{2,n})$. 

\begin{remark}
In the conformal model, the spacelike geodesics of $\AdS_{n+1}$ ending at some point $x$
of $\partial\AdS_{n+1}$ are all orthogonal to $\partial\AdS_{n+1}$ at $x$ whereas in the Klein model 
spacelike geodesics ending at a given point in $\partial\ADS_{n+1}$ are not tangent one to the other. Hence the
homeomorphism between 
$\AdS_{n+1}\cup\partial\AdS_{n+1}$ and $\ADS_{n+1}\cup\partial\ADS_{n+1}$ is not
a diffeomorphism.
\end{remark}

\begin{defi}
\label{def.affine}
For every $x$ in $\AdS_{n+1}$, the \emph{affine domain} $U(x)$ of $\ADS_{n+1}$ is the connected component of
$\ADS_{n+1}\setminus\SS(x^{\orth})$ containing $x$. Let $V(x)$ be the connected component of
$\SS(\RR^{2,n})\setminus\SS(x^{\orth})$ containing $U(x)$. The boundary $\partial
U(x) \subset\partial\ADS_{n+1}$ of $U(x)$ in $V(x)$ is called the \textit{affine 
  boundary\/} of $U(x)$. 
\end{defi}

\begin{remark}
\label{rk.affine}
Up to composition by an element of the isometry group $SO_0(2,n)$ of
$\mathrm{q}_{2,n}$, we can assume that $\SS(x^{\orth})$ is the projection of the
hyperplane $\{ u=0\} $ in $\RR^{2,n}$ and $V(x)$ is the projection of the
region $\{ u>0\} $ in $\RR^{2,n}$. The map 
$$(u, v, x_1,x_2,\dots,x_{n+1})\mapsto
(t, \bar{x}_1,\dots, \bar{x}_{n}) :=
(\frac{v}{u}, \frac{x_1}{u},\frac{x_2}{u},\dots,\frac{x_{n}}{u})$$
induces a diffeomorphism between   
$V(x)$ and $\RR^{n+1}$ mapping the affine domain $U(x)$ to the region $\{ -t^2+\bar{x}_1^2+\dots+\bar{x}_n^2 <
1\} $. The affine boundary $\partial U(x)$ corresponds to the
hyperboloid $\{ -t^2+\bar{x}_1^2  +\dots+\bar{x}_n^2=1\} $. The intersections between
$U(x)$ with the totally geodesic subspaces of $\ADS_{n+1}$ correspond to the
intersections of the region $\{ -t^2+\bar{x}_1^2+\dots+\bar{x}_n^2 <
1\} $ with the affine subspaces of $\RR^{n+1}$.
\end{remark}

Although the real number $\langle x \mid y \rangle$ is
well-defined only for $x,y\in\RR^{2,n}$, its sign is well-defined for $x,y\in\SS(\RR^{2,n})$. 

\begin{lemma}
\label{l.causally-related}
Let $U$ be an affine domain in $\ADS_{n+1}$ and $\partial
U\subset\partial\ADS_{n+1}$ be its affine boundary. Let $x$ be be a point
in $\partial U$, and $y$ be a point in $U\cup\partial U$. There exists
a causal (resp. timelike) curve joining $x$ to $y$ in $U\cup\partial
U$ if and only if $\langle x\mid y\rangle \geq 0$ (resp. $\langle
x\mid y\rangle>0$).  
\end{lemma}

\begin{proof}
See e.g. \cite[Proposition~5.10]{barbtz1} or~\cite[Proposition
4.19]{BBZ}. 
\end{proof}

\subsection{The Klein model of the Einstein universe}
Similarly, Einstein universe has a Klein model: it is
the projection $\SS(\cC_n)$ in $\SS(\RR^{2,n})$ of the
isotropic cone $\cC_n$ in $\RR^{2,n}$. 
The conformal Lorentzian structure can be defined in terms
of the quadratic form $\mathrm{q}_{2,n}$. In particular, an immediate
corollary of Lemma~\ref{l.causally-related} is:

\begin{cor}
\label{cor.list}
For $\Lambda \subseteq \Ein_n$, the following assertions are
equivalent. 
\begin{enumerate}
\item $\Lambda$ is achronal (respectively acausal)
\item when we see $\Lambda$ as a subset of $\SS(\cC_n) \approx \Ein_n$ the
scalar product $\langle x \mid y \rangle$ is non-positive
(respectively negative) for every distinct $x,y\in\Lambda$.
\end{enumerate}\fin
\end{cor}

In the sequel, we will frequently identify $\Ein_{n}$ with $\SS(\cC_n)$,
since it is common to skip from one model to the other. For more details
about the Einstein universe, see \cite{franceseinstein, primer}.

\begin{remark}
\label{rk.desitter}
The affine boundary $\partial{U}(x)$ defined in remark~\ref{rk.affine}, as a domain of $\Ein_{n}$, is conformally isometric to
the \textit{de Sitter space}. Hence we also call it de Sitter domain.
\end{remark}





\subsection{Unit tangent bundle}
\label{sub.ell}
Denote by $\cE^1 \AdS_{n+1}$ (resp. $\cL^1 \AdS_{n+1}$) the
tangent bundle of unit spacelike (respectively lightlike) tangent vectors.
For such a vector $v$ tangent to $\AdS_{n+1}$ at $x$,
the geodesic issued from $(x,v)$ has a future and past limit in the
Einstein universe. We denote by ${\ell}^\pm: \cE^1 \AdS_{n+1} \cup \cL^1 \AdS_{n+1}
\to \Ein_n$ the applications which maps such a vector to its limits.

\section{Regular $\AdS$ manifolds}

\subsection{AdS regular domains}

Let $\wt\Lambda$ be a closed achronal subset of $\partial\wt\AdS_{n+1}$, and
$\Lambda$ be the projection of $\wt\Lambda$ in $\partial\AdS_{n+1}$.  We denote
by  $\wt E(\wt\Lambda)$ the \emph{invisible domain} of $\wt\Lambda$ in
$\wt\AdS_{n+1}\cup\partial\wt\AdS_{n+1}$, that is,

$$\wt E(\wt\Lambda)=:\left(\wt\AdS_{n+1}\cup\partial\wt\AdS_{n+1}\right) \setminus
\left(J^-(\wt\Lambda)\cup J^+(\wt\Lambda)\right)
$$  
where $J^-(\wt\Lambda)$ and $J^+(\wt\Lambda)$ are the causal past and the
causal future of $\wt\Lambda$ in
$\wt\AdS_{n+1}\cup\partial\wt\AdS_{n+1}=(\RR\times\overline{\DD}^{n-1},-d\theta^2+ds^2)$.
We denote by $\Cl(\wt E(\wt\Lambda))$ the closure of $\wt E(\wt\Lambda)$ in
$\wt\AdS_{n+1}\cup\partial\wt\AdS_{n+1}$ and by $E(\Lambda)$ the projection
of $\wt E(\wt \Lambda)$ in $\AdS_{n+1}\cup\partial\AdS_{n+1}$ 
(according to Corollary~\ref{cor.lift}, $E(\Lambda)$ only depends on $\Lambda$, not on $\wt\Lambda$).

\begin{defi}
A $n$-dimensional \emph{AdS regular domain} is a domain of the form
$E(\Lambda)$ where $\Lambda$ is the projection in $\partial\AdS_{n+1}$ of an
achronal subset $\wt\Lambda\subset\partial\wt{\mbox{AdS}_{n+1}}$ containing 
at least two points.
If $\wt\Lambda$ is a topological $(n-1)$-sphere, then $E(\Lambda)$ is
\emph{GH-regular} (this definition is motivated by theorem
\ref{t.iso-quotient}). 
\end{defi}

\begin{rema}
For every closed achronal set $\wt\Lambda$ in $\partial\wt\AdS_{n+1}$, the
invisible domain  $\wt E(\wt \Lambda)$ is \emph{causally  convex} in
of $\wt\AdS_{n+1}\cup\partial\wt\AdS_{n+1}$: this is an immediate consequence 
of the definitions. It follows that AdS regular domains are strongly causal.
\end{rema}


\begin{rema}
\label{r.f-f+}
Let $\wt\Lambda$ be a closed achronal subset of $\partial\wt\AdS_{n+1}$.
Recall that $\wt\Lambda$ is the graph of a $1$-Lipschitz function
$f:\Lambda_0\to\RR$ where $\Lambda_0$ is a closed subset of
$\SS^{n-1}$ (\S~\ref{sub.achro}).  Define two functions
$f^{-},f^{+}:\overline{\DD}^{n}\to\RR$ as follows: 
\begin{align*}
f^{-}(\mathrm{x}) &:= \mbox{Sup}_{\mathrm{y} \in \Lambda_0} \{ f(\mathrm{y})-d(\mathrm{x},\mathrm{y}) \} , \\
f^{+}(\mathrm{x}) &:= \mbox{Inf}_{\mathrm{y} \in \Lambda_0} \{ f(\mathrm{y})+d(\mathrm{x},\mathrm{y}) \} ,
\end{align*}
where $d$ is the distance induced by $ds^2$ on
$\overline{\DD}^{n}$. It is easy to check that 
$$
\wt E(\wt\Lambda)=\{(\theta, \mathrm{x})\in\RR\times\overline{\DD}^{n} \mid
f^{-}(\mathrm{x})< \theta <f^{+}(\mathrm{x})\}.
$$
\end{rema}

The following lemma is a refinement of lemma~\ref{le.achrinj}:

\begin{lem}
\label{le.one-to-one}
For every (non-empty) closed achronal set
$\wt\Lambda\subset\partial\wt\AdS_{n+1}$, the projection of $\wt
E(\wt\Lambda)$ on $E(\Lambda)$ is one-to-one. 
\end{lem}

\begin{proof}
We use the notations introduced in remark~\ref{r.f-f+}. For every
$\mathrm{x}\in\overline{\DD}^{n}$, there exists 
a point $\mathrm{y}\in\SS^{n-1}=\partial\overline{\DD}^{n}$ such that 
$d(\mathrm{x},\mathrm{y})\leq \pi/2$. Hence, for every $\mathrm{x}\in\overline{\DD}^{n}$, we
have $f^{+}(\mathrm{x})-f^{-}(\mathrm{x})\leq \pi$. Hence $\wt
E(\wt\Lambda)$ lies in
$
E=\{(\theta, \mathrm{x})\in\RR\times\overline{\DD}^{n} \mbox{ such that } f^{-}(\mathrm{x}) <
\theta < f^{-}(\mathrm{x})+\pi\}. 
$
The restriction to $E$ of the projection of $\wt\AdS_{n+1}\cup\partial\wt\AdS_{n+1}=\RR\times\overline{\DD}^{n}$
on  $\AdS_{n+1}\cup\partial\AdS_{n+1}=(\RR/2\pi\ZZ)\times\overline{\DD}^{n-1}$ is
obviously one-to-one.
\end{proof}

\begin{defi}
\label{def.purelightlike}
An achronal
subset $\wt\Lambda$ of $\wt\Ein_{n+1}$ is \emph{pure lightlike} if the associated subset
$\Lambda_0$ of $\SS^{n}$ contains two antipodal points $\mathrm{x}_0$ and $\mathrm{x}_0^{\ast}$ 
such that, for the associated 1-Lipschitz map $f:\Lambda_0\to\RR$ the equality $f(\mathrm{x}_0)=f(\mathrm{x}_0^{\ast}) +\pi$ holds.
\end{defi}

If $\wt\Lambda$ is pure lightlike, for every element $\mathrm{x}$ of
$\overline{\DD}^{n}$  we have $f^{-}(\mathrm{x})=f^{+}(\mathrm{x}) =
f(\mathrm{x}_0^{\ast}) + d(\mathrm{x}_0^{\ast}, \mathrm{x})=f(\mathrm{x}_0) - d(\mathrm{x}_0, \mathrm{x})$, 
implying that $\wt E(\wt\Lambda)$ is empty.
Conversely:

\begin{lem}
\label{le.purevide}
$\wt E(\wt\Lambda)$ is empty if and only if $\wt\Lambda$ is pure lightlike. More precisely,
if for some point $\mathrm{x}$ in $\DD^{n}$ the equality $f^{+}(\mathrm{x})=f^{-}(\mathrm{x})$ holds then
$\wt\Lambda$ is pure lightlike.
\end{lem}

\begin{proof}
Assume $f^{+}(\mathrm{x})=f^{-}(\mathrm{x})$ for some $\mathrm{x}$ in $\DD^{n}$. Then, since
$\Lambda_{0}$ is compact, the upper and lower bounds are attained: there are 
$\mathrm{y}^{\pm}$ in $\Lambda_{0}$ such that:

\begin{eqnarray*}
f^{-}(\mathrm{x}) & = & f(\mathrm{y}^{-}) - d(\mathrm{x},\mathrm{y}^{-}) \\
f^{+}(\mathrm{x}) & = & f(\mathrm{y}^{+}) + d(\mathrm{x},\mathrm{y}^{+})
\end{eqnarray*}

Hence:

\[ d(\mathrm{y}^{-}, \mathrm{y}^{+}) \geq f(\mathrm{y}^{-}) - f(\mathrm{y}^{+}) = d(\mathrm{y}^{-}, \mathrm{x}) + d(\mathrm{x}, \mathrm{y}^{+}) \]

We are in the equality case of the triangular inequality. It follows that $\mathrm{x}$ belongs to a minimizing geodesic
in $\SS^{n}$ joining $\mathrm{y}^{-}$ to $\mathrm{y}^{+}$. It is possible only if $\mathrm{y}^{+}$, $\mathrm{y}^{-}$ are antipodal one
to the other, since if not the minimizing geodesic joining them is unique and contained in $\partial\DD^{n}$.
Moreover, $f(\mathrm{y}^{-})=f(\mathrm{y}^{+})+\pi$. The lemma follows.
\end{proof}

\begin{coro}
\label{c.inter-boundary}
For every achronal topological $(n-1)$-sphere
$\wt\Lambda\subset\partial\wt\AdS_{n+1}$,
\begin{enumerate}
\item $\wt E(\wt\Lambda)$ is disjoint from $\partial\wt\AdS_{n+1}$
  (\emph{ie.} it is contained in $\wt\AdS_{n+1}$);
\item $\Cl\left(\wt E(\wt\Lambda)\right)\cap\partial\wt\AdS_{n+1} =
  \wt\Lambda$, where $Cl\left(\wt E(\wt\Lambda)\right)$ denotes the closure 
  of $\wt E(\wt\Lambda)$ in $\uEin_{n+1}$.
\end{enumerate}
\end{coro}

\begin{proof}
We use the notations introduced in remark~\ref{r.f-f+}. 
Since $\wt\Lambda$ is a topological $(n-1)$-sphere, the set $\Lambda_0$
is the whole sphere $\SS^{n-1}$. For every
$\mathrm{x}\in\SS^{n-1}=\Lambda_0$, one has $f^{-}(\mathrm{x})=f^{+}(\mathrm{x})=f(\mathrm{x})$. Finally, recall
that $(\theta, \mathrm{x})\in\wt E(\wt\Lambda)$ (resp. $(\theta, \mathrm{x})\in\Cl(\wt
E(\wt\Lambda))$) if and only if $f^{-}(\mathrm{x})<\theta<f^{+}(\mathrm{x})$ (resp. $f^{-}(\mathrm{x})\leq
\theta\leq f^{+}(\mathrm{x})$). The corollary follows.
\end{proof}

\begin{remark}
It follows from item (2) of Corollary~\ref{c.inter-boundary} that the GH-regular domain 
$E(\Lambda)$ characterizes $\Lambda$, \ie invisible domains
of different achronal $(n-1)$-spheres are different. We call
$\Lambda$ the \textit{limit set of $E(\Lambda)$.}
\end{remark}

\subsection{AdS regular domains as subsets of $\ADS_{n+1}$}
The canonical homeomorphism between $\AdS_{n+1}\cup\partial\AdS_{n+1}$ and
$\ADS_{n+1}\cup\partial\ADS_{n+1}$ allows us to see AdS regular domains as
subsets of $\ADS_{n+1}$.

\begin{lem}
\label{l.affine-domain}
Let $\Lambda\subset\partial\AdS_{n+1}$ be the projection of a closed
achronal subset of $\partial\wt\AdS_{n+1}$ which is not pure lightlike. 
We see $\Lambda$ and $E(\Lambda)$ in $\ADS_{n+1}\cup\partial\ADS_{n+1}$. Then
$\Lambda$ and $E(\Lambda)$ are contained in the union $U\cup\partial
U$ of an affine domain and its affine boundary.
\end{lem}

\begin{proof} 
See \cite[Lemma~8.27]{barbtz1}. 
\end{proof}

Lemma~\ref{l.affine-domain} implies, in particular, that every AdS
regular domain is contained in an affine domain of $\ADS_{n+1}$. This
allows to visualize AdS regular domains as subsets of 
$\RR^{n+1}$ (see remark~\ref{rk.affine}).

Putting together the definition of the invisible domain $E(\Lambda)$
of a set $\Lambda\subset\partial\AdS_{n+1}$ and
Lemma~\ref{l.causally-related}, one gets:

\begin{prop}
\label{pro.proj}
Let $\Lambda\subset\partial\AdS_{n+1}$ be the projection of a closed
achronal subset of $\partial\wt\AdS_{n+1}$ which is not pure lightlike.
If we see $\Lambda$ and $E(\Lambda)$ in the Klein
model $\ADS_{n+1}\cup\partial\ADS_{n+1}$, then
$$
E(\Lambda)=\{y \in \ADS_{n+1}\cup\partial\ADS_{n+1}\mbox{ such that }\langle
y  \mid x \rangle < 0 \mbox{ for every }x\in\Lambda\}).
$$
\end{prop}

\begin{remark}
\label{rk.nice}
A nice (and important) corollary of this Proposition is that the invisible
domain $E(\Lambda)$ associated with a set $\Lambda$ is
always geodesically convex: any geodesic joining two points in
$E(\Lambda)$ is contained in $E(\Lambda)$.  
\end{remark}

\section{Globally hyperbolic AdS spacetimes}

\subsection{Cosmological time functions}
\label{s.cosmological-time}

In any spacetime $(M,g)$, one can define the \textit{cosmological time
  function\/} as follows (see \cite{cosmic}):

\begin{defi}
The cosmological time function of a spacetime $(M,g)$ is the function
$\tau:M\rightarrow [0,+\infty]$ defined by
$$\tau(x):=\mbox{Sup}\{ L(c) \mid c \in {\mathcal R}^-(x) \},$$  where
${\mathcal R}^-(x)$ is the set of past-oriented  causal curves starting at
$x$, and $L(c)$ is  the Lorentzian length of the causal curve $c$.
\end{defi}

This function  is in general very badly behaved. For example,  in the
case of Minkowski space, the cosmological time function is everywhere
infinite. 

\begin{defi}
\label{d.regular}
A spacetime $(M,g)$ is \textit{CT-regular} with cosmological time
function $\tau$ if
\begin{enumerate}
\item $M$ has \textit{finite existence time,\/}  $\tau(x) < \infty$ for
   every $x$ in $M$,
\item for every past-oriented inextendible causal curve $c: [0, +\infty)
  \rightarrow M$, $\lim_{t \to \infty} \tau(c(t))=0$.
\end{enumerate}
\end{defi}

\begin{teo}[\cite{cosmic}]
\label{teo.cosmogood}
CT-regular spacetimes are globally hyperbolic.
\end{teo}

A very nice feature of CT-regularity is that is is preserved by isometries (and thus, by Galois automorphisms):

\begin{prop}
\label{pro.cosmolift}
Let  $(\wt{M}, \tilde{g})$ be a CT-regular spacetime. Let $\Gamma$ be a discrete group of isometries of
$(\wt{M},\tilde{g})$ preserving the time orientation and without fixed points. Then, the action of $\Gamma$ on  $(\wt{M}, \tilde{g})$
is properly discontinuous. Futhermore, the quotient spacetime $(M,g)$ is CT-regular. More precisely,
if $\mathrm{p}: \wt{M} \to M$ denote the quotient map, the cosmological times $\tilde{\tau}: \wt{M} \to [0, +\infty)$
and $\tau: M \to [0, +\infty)$ satisfy:
\[
\tilde{\tau}=\tau \circ \mathrm{p} \]
\end{prop}

\begin{proof}[Sketch of proof]
$\Gamma$ clearly preserves the cosmological time and its level sets.
These level sets are metric spaces on which $\Gamma$ acts isometrically, and hence,
properly discontinuously. It follows quite easily that $\Gamma$ acts properly discontinuously
on the entire $\wt{M}$. 

The proof of the identity $\tilde{\tau}=\tau \circ \mathrm{p}$ is straightforward: it follows
from the $\Gamma$-invariance of $\tilde{\tau}$ and the fact that 
inextendible causal curves in $M$ are precisely the projections by $\mathrm{p}$ of inextendible causal
curves in $\wt{M}$.
\end{proof}

\subsection{GH-regular AdS spacetimes are CT-regular}

Let $\Lambda$ be a non-pure lightlike topological achronal $(n-1)$-sphere in $\partial\AdS_{n+1}$.

\begin{prop}
\label{pro.adsregular}
The AdS regular domain $E(\Lambda)$ is CT-regular.
\end{prop}

\begin{proof} 
Recall that $\Lambda$ is, by definition, the projection of an
achronal topological sphere $\wt\Lambda\subset
\partial\wt\AdS_{n+1}$, and that $E(\Lambda)$ is the projection of
the invisible domain $\wt E(\wt\Lambda)$ of $\wt\Lambda$ in
$\wt\AdS_{n+1}\cup\partial\wt\AdS_{n+1}$. We will prove that $\wt
E(\wt\Lambda)$ has regular cosmological time. Since the projection of
$\wt E(\wt\Lambda)$ on $E(\Lambda)$ is one-to-one
(lemma~\ref{le.one-to-one}), this will imply that 
$E(\Lambda)$ also has regular cosmological time. We denote by $\wt\tau$ the
cosmological time of $\wt E(\wt\Lambda)$.

Let $x$ be a point in $\wt E(\wt\Lambda)$. On the one hand, according to 
corollary~\ref{c.inter-boundary}, $\Cl(\wt E(\wt\Lambda))$
is a compact subset of $\wt\AdS_{n+1}\cup\partial\wt\AdS_{n+1}$, and
the intersection $\Cl(\wt
E(\wt\Lambda))\cap \partial\wt\AdS_{n+1}$ equals $\wt\Lambda$. On the
other hand, 
since $x$ is in the invisible domain of $\wt\Lambda$, the set
$J^-(x)$ is disjoint from $\wt\Lambda$. Therefore $J^-(x)\cap\Cl(\wt    
E(\wt\Lambda))$ is a compact subset of $\wt\AdS_{n+1}$. Therefore
$J^-(x)\cap\Cl(\wt E(\wt\Lambda))$ is conformally equivalent to a
compact causally convex domain in $(\RR\times\DD^{n},-d\theta^2+ds^2)$, with a
bounded conformal factor since everything is compact. It 
follows that the lengths of the past-directed causal curves starting
at $x$ contained in $\wt E(\wt\Lambda)$ is bounded (in other words,
$\wt\tau(x)$ is finite), and that, for every past-oriented
inextendible causal curve $c:[0,+\infty)\to \wt E(\wt\Lambda)$ with
$c(0)=x$, one has $\wt\tau(c(t))\to 0$ when $t\to\infty$.
This proves that $\wt E(\wt\Lambda)$ has regular cosmological time.
\end{proof}

Hence, GH-regular domains and their quotients are globally hyperbolic (see Theorem~\ref{teo.cosmogood}, 
Proposition~\ref{pro.cosmolift}).

\subsection{GHC AdS spacetimes are GH-regular}

\begin{defi}
A GH spacetime with constant curvature $-1$ is \emph{maximal}
(abbreviation MGH)  if it admits no non-surjective embedding in
another GH spacetime $N$ with constant curvature such that each Cauchy
hypersurface  in $M$ embeds in $N$ as a Cauchy hypersurface.
\end{defi}

Any GH spacetime with constant curvature $-1$ embeds in a MGH spacetime,
and this maximal extension is unique up to isometry (see~\cite{choquet}). 
Hence, the classification
of GH spacetimes with constant curvature $-1$ essentially reduces to
the classification of MGH ones.

\begin{theo}
\label{t.iso-quotient}
Every $(n+1)$-dimensional MGHC spacetime with constant curvature $-1$ is
isometric to the quotient of a GH-regular domain in $\AdS_{n+1}$ by a
torsion-free discrete subgroup of $\SO_0(2,n)$. 
\end{theo}

This theorem was proved by Mess in his celebrated
preprint \cite{mess1, mess2} (Mess only deals with the case where $n=2$, but his arguments also
apply in higher dimension). For the reader's convenience, we shall
recall the main steps of the proof (see~\cite[Corollary 
11.2]{barbtz1} for more details).   
  
\begin{proof}[Sketch of proof of Theorem~\ref{t.iso-quotient}]
Let $(M,g)$ be $(n+1)$-dimensional MGHC spacetime with constant curvature
$-1$. In other words, $(M,g)$ is a locally modeled on $\AdS_{n+1}$.
The theory of
$(G,X)$-structures provides us with a locally isometric developing
map $\cD:\widetilde M\to \wt\AdS_{n+1}$ and a holonomy representation
$\rho:\pi_1(M)\to\SO_0(2,n)$. Pick a Cauchy  
hypersurface $\Sigma$ in $M$, and a lift $\widetilde \Sigma$ of
$\Sigma$ in $\widetilde M$. Then $\wt S:=\cD(\widetilde\Sigma)$ is an
immersed complete spacelike hypersurface in $\wt\AdS_{n+1}$. One can
prove that such a hypersurface is automatically properly embedded
and corresponds to the graph of a $1$-Lipschitz function
$f:\DD^n\to\RR$ in the conformal model
$(\RR\times\DD^2,-d\theta^2+ds^2)$. Such a function extends to a
$1$-Lipschitz function $\bar f$ defined on the closed disc
$\overline{\DD}^n$. This shows that the boundary $\partial\wt S$ of $\wt S$
in $\wt\AdS_{n+1}\cup\partial\wt\AdS_{n+1}$ is an achronal topological sphere $\wt\Lambda$
contained in $\partial\wt\AdS_{n+1}$. 

On the one hand, it is easy to see that the Cauchy development
$D(\wt S)$ coincides with the invisible domain $\wt E(\wt\Lambda)$ (this
essentially relies on the fact that $\wt S\cup\partial\wt S$ is the graph of
the $1$-Lipschitz function $\bar{f}$, hence an achronal set in
$\widetilde{\mbox{AdS}}_{n+1} \cup \partial\wt\AdS_{n+1}$).  
 
On the other hand, since $\widetilde\Sigma$ is a Cauchy
hypersurface in $\widetilde M$,  the image $\cD(\widetilde M)$
is necessarily contained in $D(\wt S)=\wt E(\wt\Lambda)$.
Hence, the developing map $\cD$ induces an isometric embedding
from $M$ into $\Gamma\setminus D(\wt S)$, where $\Gamma:=\rho(\pi_1(M))$. 
Since $M$ is maximal, this embedding must be
onto, and thus, $M$ is isometric to the quotient $\Gamma\setminus 
\wt E(\wt\Lambda)$.  
\end{proof}

\begin{defi}
A representation $\rho: \Gamma \to \SO_{0}(2,n)$ is GH-regular if it is the holonomy of 
a GH-regular spacetime, \ie if is is faithfull, discrete and preserves a GH-regular
domain. If the quotient spacetime $\rho(\Gamma)\backslash{E}(\Lambda)$ is spatially compact, 
we say that $\rho$ is GHC-regular.
\end{defi}

\section{Anosov anti-de Sitter manifolds are MGHC}

\subsection{Anosov representations}
\subsubsection{General definition}
\label{sub.defsgenerales}
Let $N$ be a manifold equipped with a non singular flow $\Phi^{t}$ and an auxiliary Riemannian metric $\Vert$.

\begin{defi}
\label{defanosov}
A closed subset $F \subseteq N$ is $\Phi^{t}$-hyperbolic if it is $\Phi^{t}$-invariant 
and if the tangent bundle of $N$ admits a decomposition 
$TN=\Delta \oplus E^{ss} \oplus E^{uu}$ over $F$ such that,
for some positive constants $a$, $b$:

\begin{itemize}
\item The line bundle $\Delta$ is tangent to the flow,

\item for any vector $v$ in $E^{ss}$ over a point $p$ of $N$, and for any positive $t$:

\[ \Vert d_{p}\Phi^{t}(v) \Vert \leq b e^{-at} \Vert v \Vert \]

\item for any vector $v$ in $E^{uu}$ over a point $p$ of $N$, and for any negative $t$:

\[ \Vert d_{p}\Phi^{t}(v) \Vert \leq b e^{at} \Vert v \Vert \]

\end{itemize}
If $F$ is the entire manifold $N$, the flow $\Phi^{t}$ is Anosov.
\end{defi}

Typical examples are geodesic flows of negatively curved Riemannian manifolds. Let $\Gamma$
be the fundamental group of $N$. Let $Y$ be a manifold, and $G$ be a Lie group acting smoothly
on $Y$. Given any representation $\rho: \Gamma \to G$ one contructs the associated
flat bundle $E_{\rho}$ over $N$: it is the quotient of the product $\tilde{N} \times Y$ by
the natural action of $\Gamma$, with the projection $\pi_\rho: E_\rho
\to N$. The bundle $E_\rho$ inherits a flow $\Phi_\rho^t$ from the lifting $\tilde{\Phi}^{t}$ of
$\Phi^t$ on $\widetilde{N}$. 

\begin{defi} A representation $\rho: \Gamma \to G$ is $(G,Y)$-Anosov
if the flat bundle $\pi_\rho:E_\rho \to N$ admits a continuous section $s: N\to
  E_\rho$ such that the image of $s$ is an invariant hyperbolic subset for $\Phi_{\rho}^{t}$.
\end{defi}

A very nice feature of the Anosov representations is the following
proposition, which is consequence of the stability property of
closed hyperbolic set (see \cite{labourie}, proposition 2.1
for a proof):

\label{not:rep}
\begin{theo}
\label{thm:labourie}
Let $N$ a compact manifold endowed
with an Anosov flow $\Phi^t$. The set of $(G,Y)$-Anosov
representations from $\Gamma=\pi_1(N)$ to $G$ is open in the space
of representations of $\Gamma$ in $G$, usually denoted
$\Rep(\Gamma,G)$ (endowed with the compact-open topology)
\end{theo}

\subsubsection{Anosov AdS representations}
Here, we are concerned with the case $(G, Y)=(\SO_{0}(2,n), \cY)$ where
$\cY$ is the open subset of
$\Ein_n \times \Ein_n$ made of the pairs of points that can be
joined by a spacelike geodesic. 
Given a $(\SO_{0}(2,n), \cY)$-Anosov representation, the section $s: N \to E_{\rho}$
defining the Anosov property lifts to a map $\ell_\rho: \wt{N} \to \cY$ which is $\Gamma$-equivariant, \ie
$\rho_g \circ \ell_\rho= \ell_\rho \circ g$, and
$\Phi^t$-invariant. This application can be decomposed in
$\ell_\rho=(\ell_\rho^+, \ell_\rho^-)$ where $\ell_\rho^+$
(resp. $\ell_\rho^-$) are two applications from $\wt{N}$ to $\Ein_n$.

\begin{remark}
\label{rk.def2anosov}
An equivalent way to formulate the $(G, \cY)$-Anosov property is to require the
existence of continuous maps $\ell^{\pm}_{\rho}: \wt{N} \to \Ein_{n}$ and 
of a family of Riemannian metrics $g^{p}$ depending continuously on
$p \in \wt{N}$ and defined in a neighborhood of $\ell_\rho^{\pm}(p)$ in $\Ein_{n}$ such that:

\begin{enumerate}
\item
This family is $\Gamma$-equivariant, \ie for every $\gamma$ in $\Gamma$:

$$
\begin{aligned}
&{g}^{\gamma p}(d\rho(\gamma) w, d\rho(\gamma) w) =
{g}^{p}(w,w) \end{aligned}$$

where $w$ belongs to $\T_{\ell^\pm(p)}\Ein_n$ and $d\rho(\gamma)$ is the differential of
$\rho(\gamma)$ at $\ell^\pm(p)$.

\item
The family increases (resp. decreases) exponentially along positive (resp. negative) orbits
of $\tilde{\Phi}^{t}$, \ie for some $a, b >0$, if $w$ is a vector tangent at $\ell_\rho^{+}(p)$ (resp. $\ell_\rho^{-}(p)$) to $\Ein_{n}$,
then:

$$
\begin{aligned}
&{g}^{\tilde{\Phi}^{t}(p)}(w,w) \geq b^{-1}\exp(at)
{g}^{p}(w,w) \\
&{g}^{\tilde{\Phi}^{t}(p)}(w,w) \leq b\exp(-at)
{g}^{p}(w,w)\end{aligned}$$

\end{enumerate}
\end{remark}

\subsubsection{Basic properties of the geodesic flow}
\label{sub.progeodesic}
From now, we only consider the case 
where the Anosov flow is the geodesic flow on the unit tangent bundle 
over a hyperbolic manifold:  $\Gamma$ is a cocompact torsion free lattice of
$\SO_{0}(1,n)$, $N$ is the quotient $\T^1\HH^n$ by $\Gamma$ and $\Phi^{t}$
is the geodesic flow $\phi^{t}$ on $\Gamma\backslash\T^{1}\HH^{n}$, projection
of the geodesic flow $\tilde{\phi}^{t}$ of $\T^{1}\HH^{n}$.
Let's remind
few well-known properties of geodesic flows on hyperbolic manifolds and fix some notations:

\begin{enumerate}
\item The orbit by $\tilde{\phi}^{t}$ of $(x,v)$ in $\T^{1}\HH^{n}$ is the set of 
points $(x^{t}, v^{t})$ where $x^{t}$ describes at unit speed the geodesic tangent to $(x,v)$
and $v^{t}$ is tangent to this geodesic at $x^{t}$,

\item We denote by $\ell^+(x,v)$ the future extremity in $\partial\HH^{n}$ of the geodesic tangent to $(x,v)$, and by
$\ell^-(x,v)$ the past extremity of this geodesic. The fibers of $\ell^+$ (respectively of $\ell^-$) are called
the \textit{stable leaves} (respectively the \textit{unstable leaves}).

\item If $(x,v)$, $(x',v')$ belong to the same stable leaf, then there is some $T$ such that the hyperbolic distance between 
$\tilde{\phi}^{t+T}(x,v)$ and $\tilde{\phi}^{t}(x',v')$ exponentially tend to $0$ when $t \to +\infty$.

\item Geodesics in $\HH^{n}$ are characterized by their extremities, which are distinct.
Hence the map $(\ell^+, \ell^-): \T^{1}\HH^{n} \to \partial\HH^{n} \times \partial\HH^{n}$ induces an identification between the orbit
space of $\tilde{\phi}$ and $\partial\HH^{n} \times \partial\HH^{n} \setminus \sD$ where 
$\sD$ is the diagonal.

\item Every element $\gamma$ of $\Gamma$ is loxodromic: it admits one attractive fixed point $x^{+}_{\gamma}$ in
$\partial\HH^{n}$ and one repelling fixed point $x^{-}_{\gamma}$. The geodesic with extremities $x^{+}_{\gamma}$ 
and $x^{-}_{\gamma}$ is the unique geodesic of $\HH^{n}$ preserved by $\gamma$. There is a real number $T>0$ such
that for every $(x,v)$ tangent to the $\gamma$-invariant geodesic and such that $\ell^+(x,v)=x^{+}_{\gamma}$,
$\ell^-(x,v) =x^{-}_{\gamma}$ we have: $\tilde{\phi}^{T}(x,v)=\gamma(x,v)$.

\item Attractive fixed points of elements of $\Gamma$ are dense in $\partial\HH^{n}$.

\item Periodic geodesics are dense in $N$, \ie pairs $(x^{+}_{\gamma}, x^{-}_{\gamma})$ are dense in
$\partial\HH^{n} \times \partial\HH^{n} \setminus \sD$.

\end{enumerate}

\subsection{Fuchsian representations are Anosov}
\label{sub.inclusion}

The representation $\rho: \Gamma \to \SO_{0}(2,n)$ is \textit{Fuchsian} if it is faithfull,
discrete, and that $\rho(\Gamma)$ admits
a global fixed point in $\AdS_{n+1}$. Up to conjugacy in $\SO_{0}(2,n)$
every Fuchsian representation is the inclusion $\rho_{0}: \Gamma \subset \SO_{0}(1,n) \subset  \SO_{0}(2,n)$.
In this $\S$ we prove that Fuchsian representations are $(\SO_{0}(2,n), \cY)$-Anosov.

\subsubsection{De Sitter domains in $\Ein_n$}
For any $x$ in $\AdS_{n+1}$ the associated de Sitter domain $\partial{U}(x)$ 
(cf. remark~\ref{rk.desitter}) is the open subset 
of $\Ein_n$ comprising limits of spacelike geodesics starting at $x$.
If $(x,v)$ is a unit spacelike
tangent vector to $\AdS_{n+1}$ --- \ie $\mathrm{q}_{2,n}(x)=-1$, $\sca{x}{v} =
0$ and $\mathrm{q}_{2,n}(v)=-1$ --- then $x + v \subseteq \mathcal{C}_{n}$ is a
representant of $\ell^+(x,v)$ (see \S~\ref{sub.ell}). Hence
$\partial{U}(x)$ is simply the projection on the sphere
$\SS(\RR^{2,n})$ of  
$$\cU_x:=\{  x +v \tq v \in \{x\}^\orth, \mathrm{q}_{2,n}(v)=1 \} \subseteq
\mathcal{C}$$

An inverse map for this projection can be constructed from the
application $s_x: \RR^{2,n} \setminus \{x\}^\orth \to \RR^{2,n}$ which
maps a point $y \in \RR^{2,n} \setminus \{ x\}^\orth$ to the
unique colinear point $s_x(y)$ in $H_x=\{ z \tq \sca{z}{x}=-1\}$, \ie 
$s_x(y)=-y/\sca{y}{x}$. This map induces a diffeomorphism
$\tilde{s}_x: \partial{U}(x) \subseteq \Ein_n \to \cU_x$.

\subsubsection{Construction of the metric}
\label{sub.metricgV}
For each choice of a point $V \in \RR^{2,n}$ of norm $-1$ such that $\sca{x}{V} =
0$ we construct a metric $g^{x,V}$ on $\cU_x$ as follows. For any
choice of $\zeta \in \cU_x$ such that $\zeta=x+v$ ($v \in \T_x \AdS_{n+1}$), we
define a unit timelike tangent vector $\tau^{x,V}_\zeta$ to $\cU_x$ at $\zeta$
by 

$$\tau^{x,V}_\zeta:=\frac{V - \sca{V}{v} v}{\sca{V}{v}^2 + 1}$$

Let $\bar{g}^{x,V}_\zeta$ be the metric on $\cU_x$ obtained by
«changing the sign» of $\tau^{x,V}_\zeta$ in the metric induced by
$\mathrm{q}_{2,n}$ on $\T_\zeta \cU_x$. More precisely:

$$\bar{g}^{x,V}_\zeta(w,w):=\mathrm{q}_{2,n}(w,w) + 2 \sca{w}{\tau_\zeta^{x,V}}^2 $$

The pull-back of this metric by the section $s_x$  is a 
Riemannian metric ${g}^{x,V}=s_x^* \bar{g}^{x,V}$ on $\partial{U}(x)$ for each choice of $V$ in $\AdS_{n+1}$.

\begin{remarks}\-
\begin{enumerate}
\item
In the terminology of \cite{benbon}
the metric $g^{x,V}$ is the Wick rotation performed on the de Sitter metric 
of $\partial{U}(x)$ along the gradient of the time function $\zeta \to \sca{V}{\zeta}$.

\item 
The previous construction is $\SO_{0}(2,n)$-equivariant in the sense that
if $\gamma$ is any isometry of $\RR^{2,n}$,
$$
\begin{aligned}
&\gamma \partial{U}(x)=\partial{U}(\gamma(x))\\
\hbox{ and }
&{g}^{\gamma x, \gamma V}_{\gamma(\zeta)}(d\gamma w, d\gamma w) =
{g}^{x,V}_\zeta(w,w) \end{aligned}$$

\end{enumerate}
\end{remarks}

\subsubsection{The inclusion is Anosov}
The group $\rho_0(\Gamma)$ preserves an element $V$ of $\AdS_{n+1}$ and
the stable spacelike hypersurface $\SS(V^{\orth}) \cap \AdS_{n+1}$ isometric to $\HH^n$.
It gives a natural inclusion $\T^1\HH^n \subseteq \cE^1 \AdS_{n+1}$. 
We define the maps $\ell^{\pm}_{\rho_{0}}: \T^{1}\HH^{n} \to \Ein_{n}$
as the restrictions of $\ell^\pm$ to $\T^1\HH^n$. They are both $\tilde{\phi}^t$-invariant and
$\Gamma$-equivariant. Since $\HH^n \subseteq \AdS_{n+1}$ is spacelike,
$\ell_{\rho_0}^+(x,v)$ and $\ell_{\rho_0}^-(x,v)$ are joigned by a
spacelike geodesic, implying that the map $\ell_{\rho_0} :=
(\ell_{\rho_0}^-, \ell_{\rho_0}^+)$ takes its value in $\cY \subseteq
\Ein_n \times \Ein_n$. In order to prove that the representation
$\rho_0$ is Anosov, we only need to check the hyperbolicity property, as formulated in Remark~\ref{rk.def2anosov},
for the family of metrics ${g}^{(x,v)}:={g}^{x,V}$. It is the matter of the following proposition which only
establishes the expanding property at $\ell_{\rho_0}^{+}(x,v)$, the contracting property at $\ell_{\rho_0}^{-}(x,v)$
is similar.

\begin{prop}
\label{pro.numultiplie}
Let $(x,v)$ be an element of $\T^{1}\HH^{n}$ and $\nu$ a vector tangent to $\Ein_{n}$ at $\ell_{\rho_0}^{+}(x,v)$.
Then ${g}^{\tilde{\phi}_{\rho_{0}}^{t}(x,v)}(\nu,\nu)=\exp(2t)g^{(x,v)}(\nu,\nu)$.
\end{prop}

\begin{proof}
Let $(x,v)$ be an element of $\T^1\HH^n$, and $x^t$ be the
base-point of $\tilde{\phi}^t(x, v)$, \ie $x^t=(\cosh t) x + (\sinh t) v$.
While the limit
vector $\zeta=\ell_{\rho_0}^+(\tilde{\phi}^{t}(x,v))$ doesn't change in the Einstein universe, its
representant in $\cU_{x^t} \subset \mathcal{C}_{n} \subset \RR^{2,n}$ vary with $t$; the
exponential expanding behaviour comes from the changes in the
derivative of the maps $s_{x^t}$.

The representant of $\zeta=\ell_{\rho_0}^{+}(x,v)$ in $\cU_{x}$ is $x+v$. The tangent vector
$\nu$ is the image under the derivative of $s_{x}$ of a vector $w$ tangent to $\cU_{x} \cap \cC_n$. 
In particular, we have $\sca{w}{x}=\sca{w}{v}=0$.
Its representant in $\cU_{x^{t}}$ is the image $w^{t}$ of $w$ under the derivative 
of $s_{x^{t}}$ at $x+v$, \ie:

$$ \begin{aligned}
w^{t}=\d_{x+v} s_{x^t}(w) &=\frac{(x+v) \sca{w}{x^t} - w
  \sca{{x+v}}{x^t}}{\sca{{x+v}}{x^t}^2}\\
  &= \frac{ 0  + w e^{-t}}{e^{-2t}} \\
&= e^t  w
\end{aligned}
$$

Since $\sca{V}{v}=0$, the vector $\tau^{x^{t}, V}_{\zeta}$ is $V$ for every $t$.
The proposition follows.
\end{proof}

\subsection{Anosov representations are GH-regular}
\label{sec.anosovregular}
We proceed to the proof that Anosov representations are GH-regular. Our goal in this {\S} is to prove that the 
applications $\ell^\pm_\rho: \wt{N} \to \Ein_{n}$ have the same image
and that this common image is an acausal
hypersphere of $\Ein_{n}$. 

\begin{lemma}
The application $\ell^+_\rho$ (resp. $\ell^-_\rho$) is constant along the
leaves of the stable (resp. unstable) foliation of $\wt{N}$.
\end{lemma}

\begin{proof}
Corollary of item (3) in \S~\ref{sub.progeodesic} (and from the compactness of $\Gamma\backslash\HH^{n}$).
\end{proof}

Therefore, $\ell_\rho^{\pm}$ induce $\Gamma$-equivariant maps $\bar{\ell}_{\rho}^{\pm}: \partial\HH^{n} \to \Ein_{n}$. 

\begin{prop}
Let $\alpha$ be the map from $\T^1 \HH^n$ to itself which sends the
vector $(x,v)$ to $(x,-v)$. Then $\ell^+_\rho=\ell^-_\rho \circ
\alpha$.
\label{prop:ellconj}
\end{prop}

Before proving this proposition we need a few lemmas : 

\begin{lemma}
\label{le:attractive}
Let $\gamma$ be an element of $\Gamma$. Then $\bar{\ell}_\rho^{+}(x^{+}_{\gamma})$ (resp. $\bar{\ell}_\rho^-(x^{-}_{\gamma})$) 
is an attractive (resp. repelling) fixed point of $\rho(\gamma)$ (cf. item (5) in \S~\ref{sub.progeodesic}).
\end{lemma}

\begin{proof}
Let $(x,v) \in \T^1\HH^n$ such that $\ell^+(x,v)=x^{+}_{\gamma}$ and $\ell^-(x,v)=x^{-}_{\gamma}$. 
The images $\zeta^{\pm}=\bar{\ell}_{\rho}^{\pm}(x^{\pm}_{\gamma})=\ell_\rho^{\pm}(x,v)$ are obviously fixed points of $\rho(\gamma)$.
For some $T>0$ we have $\tilde{\phi}^{nT}(x,v)=\gamma^{nT}(x,v)$. Consider the family
of metrics $g^{(x,v)}$ appearing in the alternative definition of Anosov representations (Remark~\ref{rk.def2anosov}).
Then,  for every vector $w$ tangent to $\ell^+_\rho(x^+_{\gamma})$ in $\Ein_n$:

$$ 
g^{\tilde{\phi}^{nT}(x,v)}_{\zeta^{+}}(w,w) \geq a\exp(nT)g^{(x,v)}_{\zeta^{+}}(w,w) 
$$

On the other hand, since this family of metrics is $\Gamma$-equivariant:

$$
g^{\tilde{\phi}^{nT}(x,v)}_{\zeta^{+}}(w,w)=g^{\gamma^{n}(x,v)}_{\zeta^{+}}(w,w)
=g^{(x,v)}(d_{\zeta^{+}}\rho(\gamma)^{-n}(w), d_{\zeta^{+}}\rho(\gamma)^{-n}(w))
$$

Hence $\zeta^{+}$ is an attractive fixed point. Similarly, $\zeta^{-}$ is repelling.
\end{proof}

\begin{lemma}
\label{le:attractiveone}
The image $\rho(\gamma)$ of an element of $\Gamma$ admits exactly one attractive fixed point in $\Ein_{n}$.
\end{lemma}

\begin{proof}
Lex $x^+$ be an attractive fixed point of $\rho(\gamma)$ in $\Ein_n$ :
there exists a neighborhood $U$ of $x^+$ in $\Ein_n$ such that for all 
$y  \in U$, $\rho(\gamma)^n y \rightarrow x^+$. The convex hull
of $U$ in $\PP(\RR^{2,n})$ satisfies the same property, but it is also a neighborhood of $x^{+}$ in $\PP(\RR^{2,n})$
(it follows from the fact that in any affine chard of $\PP(\RR^{2,n})$ around $x$ the
Einstein space is a one sheet hyperboloid). Hence $x^{+}$
 is also an attractive fixed point in the projective space $\PP(\RR^{2,n})$.
 The lemma follows since attractive fixed points of projective automorphisms of $\PP(\RR^{2,n})$ are unique.
\end{proof}

\begin{proof}[Proof of proposition~\ref{prop:ellconj}]
It follows from lemmas~\ref{le:attractive}, \ref{le:attractiveone}
that this proposition is true when the geodesic tangent
to $(x,v)$ is preserved by a non trivial element of $\Gamma$. The
general case follows from the density of periodic orbits (item (7) in
\S~\ref{sub.progeodesic}). 
\end{proof}

\begin{coro}
The applications $\ell^\pm_{\rho}$ have the same image. They are homeomorphisms between
$\partial\HH^{n}$ and a topological acausal $(n-1)$-sphere $\Lambda_\rho$.
\end{coro}

\begin{proof}
The equality of the images is an immediate consequence of 
proposition~\ref{prop:ellconj}. We only have to show that the application $\ell^+_\rho$
(for example) is injective. Let $(x,v)$ and $(y,w)$ be two points of
$\T^1\HH^n$, belonging to two different stable leaves. Hence, there
exists a point $(z,\nu)$ which is in $(x,v)$'s stable leaf and $(y,w)$'s unstable
one. We thus have $(\ell_\rho^+(x,v),\ell_\rho^-(\alpha(y,w)))=(\ell_\rho^+(z,\nu)),
\ell_\rho^-(z,\nu)) \subseteq \mcal{Y}$. In particular, $\ell_\rho^+(x,v)$ and
$\ell_\rho^-(\alpha(y,w))=\ell_\rho^+(y,w)$ are joined by a spacelike geodesic and
must be different. 
\end{proof}

\begin{proof}[Proof of Theorem~\ref{teo:maintheorem}.]
The fact that Anosov representations are GH-regular
with acausal limit set follows from the last corollary and propositions
\ref{pro.adsregular} and \ref{pro.cosmolift}.
\end{proof}

\bibliography{artads}
\bibliographystyle{amsalpha}

\end{document}